\documentclass[11pt,psamsfonts]{amsart}
\usepackage{amsmath}
\usepackage{amsthm}
\usepackage{amssymb}
\usepackage{amscd}
\usepackage{amsfonts}
\usepackage{amsbsy}
\usepackage[latin1]{inputenc}
\usepackage{epsfig,afterpage}
\usepackage[dvips]{psfrag}
\usepackage{psfrag}
\usepackage[all]{xy}
\usepackage{color}

\newcommand{\R}{\ensuremath{\mathbb{R}}}

\newcommand{\N}{\ensuremath{\mathbb{N}}}

\newcommand{\Z}{\ensuremath{\mathbb{Z}}}

\usepackage{overpic}

\newtheorem {theorem} {Theorem}
\newtheorem {proposition}  {Proposition}

\newtheorem {definition} [theorem] {Definition}
\newtheorem {remark} {Remark}
\newtheorem {example} {Example}

\newtheorem {corollary} {Corollary}

\begin{document}

\title[Minimal Sets and Poincaré-Bendixson Theorem in NSVF] {On Poincaré-Bendixson Theorem and Non-Trivial Minimal Sets in Planar Nonsmooth Vector Fields}

\author[C.A. Buzzi, T. de Carvalho and R.D. Euzébio]
{Claudio A. Buzzi$^1$, Tiago de Carvalho$^2$ and\\ Rodrigo D. Euzébio$^1$}

\address{$^1$ IBILCE--UNESP, CEP 15054--000,
S. J. Rio Preto, S\~ao Paulo, Brazil}

\address{$^2$ FC--UNESP, CEP 17033--360,
Bauru, São Paulo, Brazil}

\email{buzzi@ibilce.unesp.br}

\email{tcarvalho@fc.unesp.br}

\email{rodrigo.euzebio@sjrp.unesp.br}


\date{}
\dedicatory{} \maketitle

\begin{abstract}
In this paper some qualitative and geometric aspects of nonsmooth vector fields theory are discussed. In the class of nonsmooth systems, that do not present sliding regions, a Poincaré-Bendixson Theorem is presented. A minimal set in planar Filippov systems not predicted in classical Poincaré-Bendixson theory and whose interior is non-empty is exhibited. The concepts of limit sets, recurrence and minimal sets for nonsmooth systems are defined and compared with the classical ones. Moreover some differences between them are pointed out.
\end{abstract}

\section{Introduction}\label{secao introducao}

\subsection{Setting the problem}\label{secao colocacao do problema}

Nonsmooth vector fields (NSVFs, for short) have become
certainly one of the common frontiers between Ma\-the\-matics and
Physics or Engineering.
Many authors have contributed to the study of NSVFs  (see
for instance the pioneering work \cite{Fi} or the didactic works \cite{diBernardo-livro,Marco-enciclopedia}, and references therein about details of
these multi-valued vector fields).  In our approach Filippov's convention is considered. So, the vector field of the model is discontinuous across a \textit{switching manifold} and it is possible for its trajectories to be
confined onto the switching manifold itself. The occurrence of such
behavior, known as sliding motion, has been reported in a
wide range of applications. We can find important examples in electrical circuits having switches, in mechanical devices in which components collide into each other, in problems with friction, sliding or squealing, among others.

For planar smooth vector fields there is a very developed theory nowadays. This theory is based in some important results. A now exhaustive list of such results include: \textit{The Existence and Uniqueness Theorem}, \textit{Hartman-Grobman Theorem}, \textit{Poincaré-Bendixson Theorem} and \textit{The Peixoto Theorem} among others. A very interesting and useful subject is to answer if these results are true or not at the NSVFs scenario. It is already known that the first theorem is not true (see Example \ref{exemplo varios conj limite} and Figure \ref{orbits} below) and the last theorem is true (under suitable conditions, see \cite{Soto-Ana}). Another extension to NSVFs of classical results on planar smooth vector fields include the concept of \textit{Poincaré Index} of a vector field in relation to a curve, as stated in \cite{Eu-canard-cycles}.

The specific topic addressed in this paper concern with a Poincaré-Ben\-dix\-son Theorem for NSVFs.
In smooth vector fields, under relatively weak hypothesis, Poincaré-Bendixson Theorem tells us which kind of limit set can arise on an open region of the Euclidean space $\mathbb{R}^{2}$. In particular, minimal sets in smooth vector fields are contained in the limit sets (this fact can not be observed in NSVFs as we show below). As far as we know, in the context of NSVFs, this theme has not been treated in the literature until now.

The paper is organized as follows: In Subsection \ref{secao enunciados} the main results are stated. In Section \ref{secao preliminares}  some of the standard theory on NSVFs, a brief introduction about Filippov systems and new definitions on this scenario are presented. In Section \ref{secao minimal} an example of a non-trivial minimal set of NSVFs is exhibited. Section \ref{secao teo poincare-bendixson} is devoted to prove the main results of the paper and discuss some aspects on it. In particular we prove a theorem very similar to the Poincaré-Bendixson Theorem to the case where sliding motion is not allowed. We also show that can occur a lot of nonstandard phenomena by considering escaping and sliding points on the switching manifold. Some of them are pointed out and compared with the classical theory of smooth vector fields. Moreover, the example given in Section \ref{secao minimal} ensures that a Poincaré-Bendixson Theorem for NSVFs with sliding motion can not be stated without extra hypothesis.

\subsection{Statement of the main results}\label{secao enunciados}
In this paper we are concerned with minimal sets and limit sets of NSVFs on the plane. For the classical theory it is well known the Poincaré-Bendixson Theorem that establishes that the limit sets of a smooth vector field is either an equilibrium point or a periodic orbit or a graph. In the main result of our paper we have an analogous result for NSVFs without sliding motion. In fact, in this case we add to the classical limit sets a pseudo-graph, a pseudo-cycle and an equilibrium point under the switching manifold (for details, see Section \ref{secao preliminares}).

Let $V$ be an arbitrarily small neighborhood of $0\in\R^2$. We
consider a codimension one manifold $\Sigma$ of $\R^2$ given by
$\Sigma =f^{-1}(0),$ where $f:V\rightarrow \R$ is a smooth function
having $0\in \R$ as a regular value (i.e. $\nabla f(p)\neq 0$, for
any $p\in f^{-1}({0}))$. We call $\Sigma$ the \textit{switching
manifold} that is the separating boundary of the regions
$\Sigma^+=\{q\in V \, | \, f(q) \geq 0\}$ and $\Sigma^-=\{q \in V \,
| \, f(q)\leq 0\}$. We can assume, locally around the origin of $\R^2$, that
$f(x,y)=y.$

Designate by $\chi$ the space of C$^r$-vector fields on
$V\subset\R^2$, with $r \geq 1$
large enough for our purposes. Call $\Omega$ the space of vector
fields $Z: V \rightarrow \R ^{2}$ such that
\begin{equation}\label{eq Z}
 Z(x,y)=\left\{\begin{array}{l} X(x,y),\quad $for$ \quad (x,y) \in
\Sigma^+,\\ Y(x,y),\quad $for$ \quad (x,y) \in \Sigma^-,
\end{array}\right.
\end{equation}
where $X=(X_1,X_2) , Y = (Y_1,Y_2) \in \chi$.  The trajectories of
$Z$ are solutions of  ${\dot q}=Z(q)$ and we accept it to be
multi-valued at points of $\Sigma$. The basic results of
differential equations in this context were stated by Filippov in
\cite{Fi}.

In the sequel we state the main results of the paper. They deal with limit sets of trajectories and limit sets of points. For a precise definition of these objects see Definition \ref{limitset}.

\begin{theorem}\label{PB}
Let $Z=(X,Y) \in \Omega$. Assume that $Z$ does not have sliding motion and it has a global trajectory $\Gamma_{Z}(t,p)$ whose positive trajectory $\Gamma^{+}_{Z}(t,p)$ is contained in a compact subset $K \subset V$. Suppose also that $X$ and $Y$ have a finite number of critical points in $K$, no one of them in $\Sigma$, and a finite number of tangency points with $\Sigma$. Then, the $\omega$-limit set $\omega(\Gamma_{Z}(t,p))$ of $\Gamma_{Z}(t,p)$ is one of the following objects:
\begin{itemize}
\item[(i)] an equilibrium of $X$ or $Y$;
\item[(ii)] a periodic orbit of $X$ or $Y$;
\item[(iii)] a graph of $X$ or $Y$;
\item[(iv)] a pseudo-cycle;
\item[(v)] a pseudo-graph;
\item[(vi)] a s-singular tangency.
\end{itemize}
\end{theorem}

As consequence, since the uniqueness of orbits and trajectories passing through a point is not achieved, we have the following corollary (see definitions of orbits and trajectories in Section \ref{secao preliminares}):

\begin{corollary}\label{corolario teo Poinc-Bend}
Under the same hypothesis of Theorem \ref{PB} the $\omega$-limit set $\omega(p)$ of a point $p\in V$ is one of the objects described in itens (i), (ii), (iii), (iv), (v) and (vi) or a union of them.
\end{corollary}

The same holds for the $\alpha$-limit set, reversing time.

For the general case where sliding motion is allowed in $\Sigma$, we can not exhibit an analogous result. In fact, as shown in Example \ref{exemplo1} and in Propositions \ref{minimal} and \ref{proposicao trajetoria nao densa}, there exist \textit{non-trivial minimal sets} (i.e., minimal sets distinct of an equilibrium point or of a closed trajectory) in this scenario.

\section{Preliminaries}\label{secao preliminares}

Consider Lie derivatives \[X.f(p)=\left\langle \nabla f(p),
X(p)\right\rangle \quad \mbox{ and } \quad X^i.f(p)=\left\langle
\nabla X^{i-1}. f(p), X(p)\right\rangle, i\geq 2
\]
where $\langle . , . \rangle$ is the usual inner product in $\R^2$.

We  distinguish the following regions on the discontinuity set
$\Sigma$:
\begin{itemize}
\item [(i)]$\Sigma^c\subseteq\Sigma$ is the \textit{sewing region} if
$(X.f)(Y.f)>0$ on $\Sigma^c$ .
\item [(ii)]$\Sigma^e\subseteq\Sigma$ is the \textit{escaping region} if
$(X.f)>0$ and $(Y.f)<0$ on $\Sigma^e$.
\item [(iii)]$\Sigma^s\subseteq\Sigma$ is the \textit{sliding region} if
$(X.f)<0$ and $(Y.f)>0$ on $\Sigma^s$.
\end{itemize}

The \textit{sliding vector field}
associated to $Z\in \Omega$ is the vector field  $Z^s$ tangent to $\Sigma^s$
and defined at $q\in \Sigma^s$ by $Z^s(q)=m-q$ with $m$ being the
point of the segment joining $q+X(q)$ and $q+Y(q)$ such that $m-q$
is tangent to $\Sigma^s$ (see Figure \ref{fig def filipov}). It is
clear that if $q\in \Sigma^s$ then $q\in \Sigma^e$ for $-Z$ and then
we  can define the {\it escaping vector field} on $\Sigma^e$
associated to $Z$ by $Z^e=-(-Z)^s$. In what follows we use the
notation $Z^\Sigma$ for both cases.  In our pictures
we represent the dynamics of $Z^\Sigma$ by double arrows.\\

\begin{figure}[!h]
\begin{center}
\psfrag{A}{$q$} \psfrag{B}{$q + Y(q)$} \psfrag{C}{$q + X(q)$}
\psfrag{D}{} \psfrag{E}{\hspace{1cm}$Z^\Sigma(q)$}
\psfrag{F}{\hspace{.7cm}$\Sigma^s$} \psfrag{G}{} \epsfxsize=5.5cm
\epsfbox{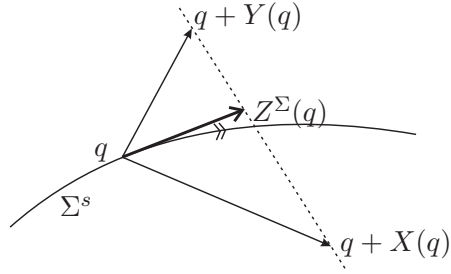} \caption{\small{Filippov's
convention.}} \label{fig def filipov}
\end{center}
\end{figure}

We say that $q\in\Sigma$ is a \textit{$\Sigma$-regular point} if
\begin{itemize}
\item [(i)] $(X.f(q))(Y.f(q))>0$
or
\item [(ii)] $(X.f(q))(Y.f(q))<0$ and $Z^{\Sigma}(q)\neq0$
(i.e., $q\in\Sigma^e\cup\Sigma^s$ and it is not an equilibrium
point of $Z^{\Sigma}$).\end{itemize}

The points of $\Sigma$ which are not $\Sigma$-regular are called
\textit{$\Sigma$-singular}. We distinguish two subsets in the set of
$\Sigma$-singular points: $\Sigma^t$ and $\Sigma^p$. Any $q \in
\Sigma^p$ is called a \textit{pseudo-equilibrium of $Z$} and it is
characterized by $Z^{\Sigma}(q)=0$. Any $q \in \Sigma^t$ is called a
\textit{tangential singularity} (or also \textit{tangency point}) and it is characterized by $(X.f(q))(Y.f(q)) =0$ ($q$ is a tangent contact
point between the trajectories of $X$ and/or $Y$ with $\Sigma$).

For a given $W\in\chi$, we say that $r$ is the \textit{contact order} of the trajectory $\Gamma_W$ of $W$ with $\Sigma$ at $p$ if $W^{k}f(p)=0$, $\forall k=0,\ldots,r-1$ and $Wf^{r}(p)\neq0$. For $W=X$ (respec. $Y$) we say that $p\in\Sigma$ is {\it invisible tangency} if the contact order $r$ of $\Gamma_X$ (respec. $\Gamma_Y$) passing through $p$ is even and $Xf^{r}(p)<0$ (respec. $Yf^{r}(p)>0$). On the other hand, for $W=X$ (respec. $Y$) we say that $p\in\Sigma$ is {\it visible tangency} if the contact order $r$ of $\Gamma_X$ (respec. $\Gamma_Y$) passing through $p$ is even and $Xf^{r}(p)>0$ (respec. $Yf^{r}(p)<0$).

A tangential singularity $p\in\Sigma^t$ is \textit{singular} if $p$ is a invisible tangency for both $X$ and $Y$.  On the other hand, a tangential singularity $p\in\Sigma^t$ is \textit{regular} if it is not singular. Figures $\ref{tangenciaregular}$ and $\ref{tangenciasingular}$ illustrate all possible cases for regular and singular tangencies, respectively.

\begin{figure}
\begin{center}
\begin{overpic}[height=0.26\paperheight]{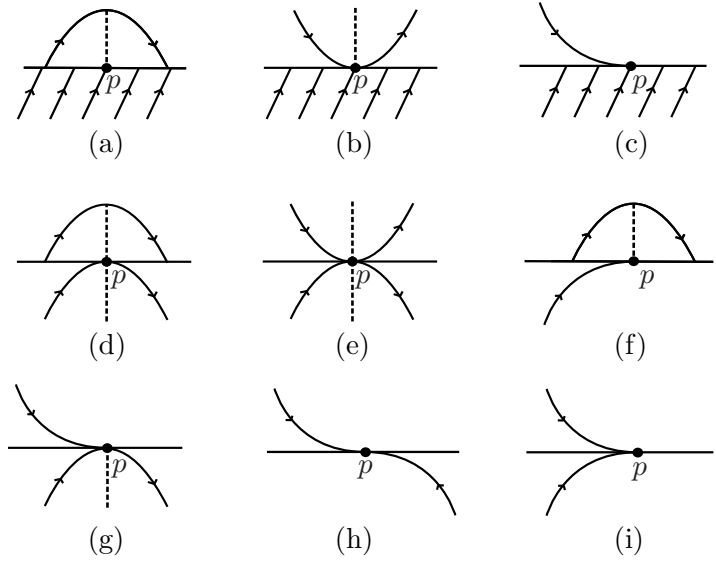}
\put(14.5,54){\color{black}{(a)}} \put(14.5,26.4){\color{black}{(d)}} \put(14.5,0){\color{black}{(g)}}
\put(48.2,54){\color{black}{(b)}} \put(48.2,26.4){\color{black}{(e)}} \put(48.2,0){\color{black}{(h)}}
\put(86.3,54){\color{black}{(c)}} \put(86.3,26.4){\color{black}{(f)}} \put(86.3,0){\color{black}{(i)}}
\put(17,62.7){$p$} \put(50.7,62.7){$p$} \put(88.8,62.7){$p$}
\put(17.7,36.1){$p$} \put(51.1,36.1){$p$} \put(89.8,36.1){$p$}
\put(17.8,10.4){$p$} \put(51.4,10.1){$p$} \put(88.8,9.9){$p$}
\end{overpic}
\caption{\small{Cases where occur regular tangential singularities. The dashed lines represent the curves where $Xf(p)=0$ or $Yf(p)=0$}.} \label{tangenciaregular}
\end{center}
\end{figure}

\begin{figure}
\begin{center}
\begin{overpic}[height=0.07\paperheight]{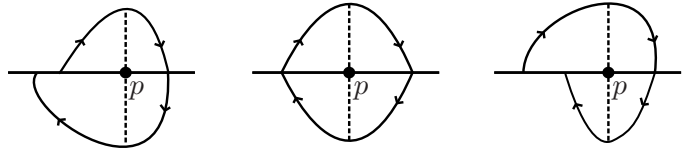}
\put(51.3,8.2){$p$} \put(89.7,8.2){$p$} \put(18,8.2){$p$}
\end{overpic}
\caption{\small{The particular cases where occur singular tangential singularities.}} \label{tangenciasingular}
\end{center}
\end{figure}

\begin{remark}
Let $p$ be in $\Sigma$ and $\Gamma_X$ (respec. $\Gamma_Y$) be the trajectory of $X$ (respec. $Y$) passing through $p$. Consider $V_{p}=V_{p}^{-}\cup\{p\}\cup V_{p}^{+}$, where $V_{p}^{-}=\{x\in\Sigma;x<p\}$ and $V_{p}^{+}=\{x\in\Sigma;x>p\}$. Let $m$ be the sum of the contact order of the trajectories $\Gamma_X$ of $X$ and $\Gamma_Y$ of $Y$ with $\Sigma$ at $p$.  It is possible to give a characterization of the behavior of $Z \in \Omega$ in the neighborhood $V_{p}$ of $p$ in terms of $m$. In fact, if $m$ is odd, then $V_{p}^{-}\subset\Sigma^{c}$ and $V_{p}^{+}\subset\Sigma^{s}\cup\Sigma^{e}$ (or vice versa, depending on the orientation). See Figure $\ref{tangenciaregular}$, items (a), (b), (f) and (g). On the other hand, if $m$ is even, we have three cases: (i) $V_{p}\setminus\{p\}$ is contained in $\Sigma^{c}$; (ii) $V_{p}\setminus\{p\}$ is contained either in $\Sigma^{s}$ or $\Sigma^{e}$; (iii) $V_{p}^{-}\subset\Sigma^{s}$ and $V_{p}^{+}\subset\Sigma^{e}$ or vice versa. This cases are represented in the Figures \ref{tangenciasingular} and \ref{tangenciaregular}, items (c), (d), (e), (h) and (i).
\end{remark}

Let $W\in\chi$. Then we denote its flow by $\phi_{W}(t,p)$. Thus,
$$
\left\{
  \begin{array}{ll}
    \dfrac{d}{dt}\phi_{W}(t,p) = W(\phi_{W}(t,p)),\\\\
    \phi_{W}(0,p)=p,
  \end{array}
\right.
$$
where $t \in I= I(p,W)\subset \R$, an interval depending on $p \in
V$ and $W$.

\smallskip

\begin{definition}\label{definicao trajetorias}
The \textbf{local trajectory (orbit)} $\phi_{Z}(t,p)$ of a NSVF given by \eqref{eq Z} is
defined as follows:
\begin{itemize}
\item For $p \in \Sigma^+ \backslash \Sigma$ and $p \in \Sigma^{-} \backslash \Sigma$ the trajectory is
given by $\phi_{Z}(t,p)=\phi_{X}(t,p)$ and
$\phi_{Z}(t,p)=\phi_{Y}(t,p)$ respectively, where $t \in I$.

\item For $p \in \Sigma^{c}$ such that $X.f(p)>0$, $Y.f(p)>0$ and  taking the
origin of time at $p$, the trajectory is defined as
$\phi_{Z}(t,p)=\phi_{Y}(t,p)$ for $t \in I \cap \{ t \leq 0 \}$ and
$\phi_{Z}(t,p)=\phi_{X}(t,p)$ for $t \in I \cap \{ t \geq 0 \}$. For
the case $X.f(p)<0$ and $Y.f(p)<0$  the definition is the same
reversing time.

\item For $p \in \Sigma^e$ and  taking the
origin of time at $p$, the trajectory is defined as
$\phi_{Z}(t,p)=\phi_{Z^{\Sigma}}(t,p)$ for $t \in I \cap \{ t \leq 0 \}$ and
$\phi_{Z}(t,p)$ is either $\phi_{X}(t,p)$ or $\phi_{Y}(t,p)$ or $\phi_{Z^{\Sigma}}(t,p)$ for $t \in I \cap \{ t \geq 0 \}$. For
the case $p \in \Sigma^s$  the definition is the same
reversing time.

\item For $p$ a regular tangency point and  taking the
origin of time at $p$, the trajectory is defined as
$\phi_{Z}(t,p)=\phi_{1}(t,p)$ for $t \in I \cap \{ t \leq 0 \}$ and
$\phi_{Z}(t,p)=\phi_{2}(t,p)$ for $t \in I \cap \{ t \geq 0 \}$, where each $\phi_{1},\phi_{2}$ is either $\phi_{X}$ or $\phi_{Y}$ or $\phi_{Z^{\Sigma}}$.

\item For $p$ a singular tangency point
    $\phi_{Z}(t,p)=p$ for all $t \in \R$.
\end{itemize}
\end{definition}

\begin{definition}\label{definicao orbita}
A \textbf{global trajectory (orbit)} $\Gamma_{Z}(t,p_0)$ of $Z\in \chi$ passing through $p_0$ is a union $$\Gamma_{Z}(t,p_0)=\bigcup_{i\in \Z} \{ \sigma_i(t,p_i); t_i \leq t \leq t_{i+1} \}$$of preserving-orientation local trajectories $\sigma_i(t,p_i)$ satisfying $\sigma_i(t_{i+1},p_i)=\sigma_{i+1}(t_{i+1},p_{i+1})=p_{i+1}$ and $t_i \rightarrow \pm \infty$ as $i \rightarrow \pm \infty$. A global trajectory is a \textbf{positive} (respectively, \textbf{negative}) \textbf{global trajectory} if $i \in \N$ (respectively,  $-i \in \N$) and $t_0 = 0$.
\end{definition}

\begin{definition}\label{limitset}
Given $\Gamma_{Z}(t,p_0)$ a global trajectory, the set $\omega(\Gamma_{Z}(t,p_0)) = \{ q \in V; \exists \, (t_{n}) \mbox{ satisfying } \Gamma_{Z}(t_n,p_0) \rightarrow q \mbox{ when } t\rightarrow + \infty \}$ (respectively \linebreak $\alpha(\Gamma_{Z}(t,p_0)) = \{ q \in V; \exists \, (t_{n}) \mbox{ satisfying } \Gamma_{Z}(t_n,p_0) \rightarrow q \mbox{ when } t\rightarrow - \infty \}$) is called $\mathbf{\omega}$\textbf{-limit} (respectively $\mathbf{\alpha}$\textbf{-limit}) \textbf{set of} $\mathbf{\Gamma_{Z}(t,p_0)}$. The $\omega$-limit (respectively $\alpha$-limit) set \textbf{of a point} $\mathbf{p}$ is the union of the $\omega$-limit (respectively $\alpha$-limit) sets of all global trajectories passing through $p$.\end{definition}

\begin{example}\label{exemplo varios conj limite}
Consider Figure \ref{orbits}. We observe that the global orbit passing through $q\in\Sigma$ is not necessarily unique. In fact, according to the third bullet of Definition \ref{definicao trajetorias}, the positive local trajectory by the point $q\in\Sigma$ can provide three distinct paths, namely, $\Gamma_{1}$, $\Gamma_{2}$ and $\Gamma_{3}$. In particular, it is clear that the Existence and Uniqueness Theorem is not true in the scenario of NSVFs. Moreover, the $\omega$-limit set of $\Gamma_{i}$, $i=1,2,3$ is, respectively, a focus, a pseudo-equilibrium and a limit cycle and, consequently, the $\omega$-limit set of $q$ being the union of these objects is not a connected set. This fact is not predicted in the classical theory. Note that the $\alpha$-limit set of $q$ is a connected set composed by the pseudo-equilibrium $p$.

\begin{figure}[!h]
\begin{center}
\begin{overpic}[height=5cm]{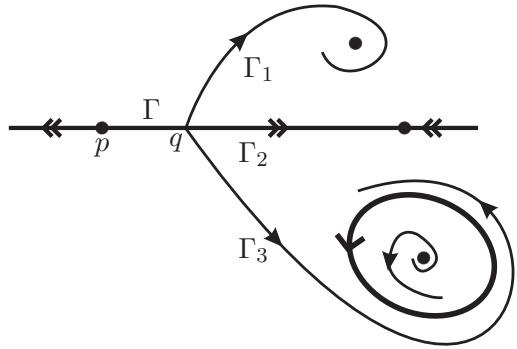}
\put(21,43){$p$}
\put(35,44){$q$}
\put(49,57){$\Gamma_{1}$}
\put(48,41){$\Gamma_{2}$}
\put(48,23){$\Gamma_{3}$}
\put(30,49){$\Gamma$}
\end{overpic}
\caption{\small{An orbit by a point is not necessarily unique.}} \label{orbits}
\end{center}
\end{figure}

\end{example}

\begin{definition}\label{definicao pseudo cycle} Consider
 $Z=(X,Y) \in \Omega.$ A closed global orbit $\Delta$ of $Z$ is a:
 \begin{itemize}
 \item[(i)] \textbf{pseudo-cycle} if $\Delta \cap \Sigma \neq
\emptyset$ and it does not contain neither equilibrium nor pseudo-equilibrium (See Figure \ref{fig canard}).
 \item[(ii)] \textbf{pseudo-graph} if $\Delta \cap \Sigma \neq
\emptyset$ and it is a union of equilibria,
pseudo equilibria and orbit-arcs of
$Z$ joining these points (See Figure \ref{fig grafico}).
 \end{itemize}
\end{definition}

\begin{figure}[h!]
\begin{minipage}[b]{0.311\linewidth}
\begin{center}
\psfrag{A}{$\Gamma$}
\epsfxsize=3.3cm \epsfbox{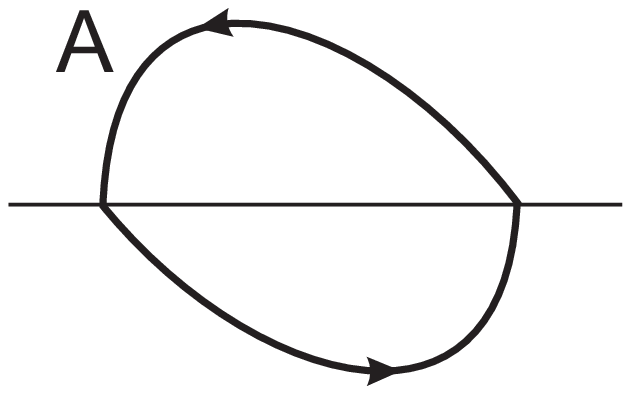}
\end{center}
\end{minipage} \hfill
\begin{minipage}[b]{0.311\linewidth}
\begin{center}
\psfrag{A}{$\Sigma=\Gamma$}\epsfxsize=2.6cm
\epsfbox{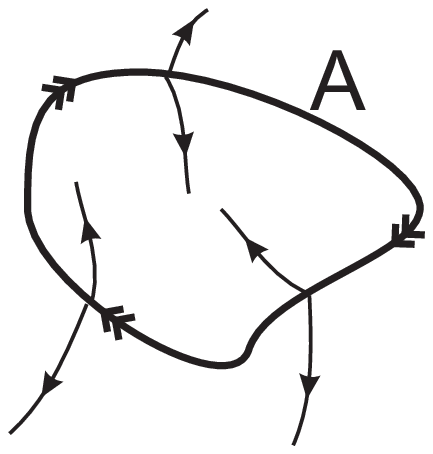}
\end{center}
\end{minipage} \hfill
\begin{minipage}[b]{0.35\linewidth}
\begin{center}
\psfrag{A}{$\Gamma$}
\epsfxsize=4cm  \epsfbox{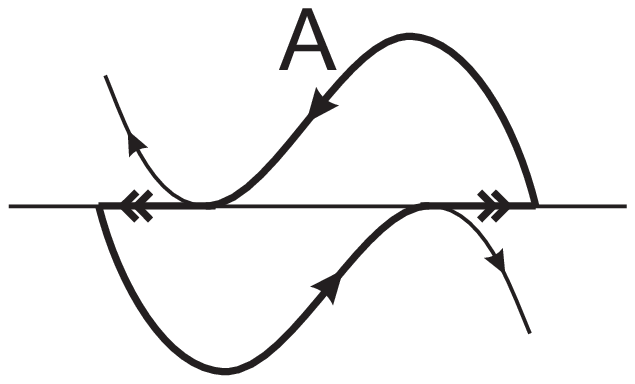}
\end{center}
\end{minipage}
\caption{\small{Possible kinds of pseudo-cycles.}}\label{fig canard}
\end{figure}

\begin{figure}[h!]
\begin{center}
\epsfxsize=11cm \epsfbox{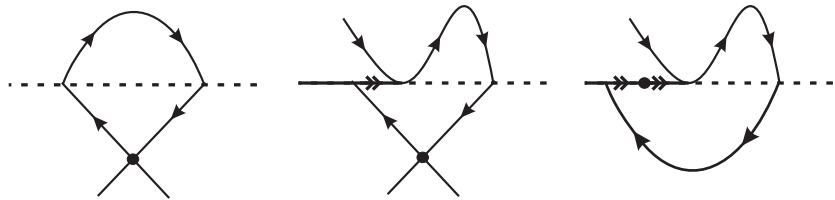}
\end{center}
\caption{\small{Examples of pseudo-graphs.}}\label{fig grafico}
\end{figure}

\begin{definition}\label{invariancia}
A set $A\subset\mathbb{R}^{2}$ is $\mathbf{Z}$\textbf{-invariant} if for each $p\in A$ and all global trajectory $\Gamma_{Z}(t,p)$ passing through $p$ it holds $\Gamma_{Z}(t,p) \subset A$.
\end{definition}

\begin{definition}
A set $M\subset\mathbb{R}^{2}$ is \textbf{minimal for} $\mathbf{Z \in \Omega}$ if
\begin{enumerate}
\item[(i)] $M\neq\emptyset$;
\item[(ii)] $M$ is compact;
\item[(iii)] $M$ is $Z$-invariant;
\item[(iv)] $M$ does not contain proper subset satisfying (i), (ii) and (iii).
\end{enumerate}
\end{definition}

\begin{remark}\label{remark 2}
Observe that the pseudo-cycle $\Gamma$ on the right of Figure 5 is the $\alpha$-limit set  of all global trajectories on a neighborhood of it, however $\Gamma$ is not $Z$-invariant according to Definition \ref{invariancia}. This phenomenon point out a distinct and amazing aspect  not predicted for the classical theory about smooth vector fields where the $\alpha$ and $\omega$-limit sets are invariant sets.
\end{remark}

\section{Minimal Sets with non-empty interior}\label{secao minimal}

Finding limit sets of trajectories of vector fields is one of the most important tasks of the qualitative theory of dynamical systems. In the literature there are several recent papers (see for instance \cite{Eu-fold-cusp,Eu-fold-sela,Marcel,Kuznetsov}) where the authors explicitly exhibit the phase portraits of some NSVFs with their unfoldings. However, all the limit sets exhibited have trivial minimal sets (i.e., the minimal sets are equilibria, pseudo equilibria, cycles or pseudo-cycles). At this section we present a non-trivial minimal set in the NSVFs scenario.

\begin{example}\label{exemplo1}
Consider $Z=(X,Y) \in \Omega$, where $X(x,y)=(1,-2x)$, $Y(x,y)=(-2,4x^{3}-2x)$ and $\Sigma=f^{-1}(0)=\{(x,y)\in\mathbb{R}^{2};y=0\}$. The parametric equation for the integral curves of $X$ and $Y$ with initial conditions $(x(0),y(0))=(0,k_{+})$ and $(x(0),y(0))=(0,k_{-})$, respectively, are known and its algebraic expressions are given by $y=-x^{2}+k_{+}$ and $y=x^{4}/2-x^{2}/2+k_{-}$, respectively. It is easy to see that $p=(0,0)$ is an invisible tangency point of $X$ and a visible one of $Y$. It is also easy to note that the points $p_{\pm}=({\pm}\sqrt{2}/2,0)$ are both invisible tangency points of $Y$. Note that between $p_{-}$ and $p$ there exists an escaping region and between $p$ and $p_{+}$ a sliding one. Further, every point between $(-1,0)$ and $p_{-}$ or between $p_{+}$ and $(1,0)$ belong to a sewing region. Consider now the particular trajectories of $X$ and $Y$ for the cases when $k_{+}=1$ and $k_{-}=0$, respectively. These particular curves delimit a bounded region of plane that we call $\Lambda$ and it is the main object of this section. Figure \ref{fig minimal} summarizes these facts.

\begin{figure}
\begin{center}
\begin{overpic}[height=0.20\paperheight]{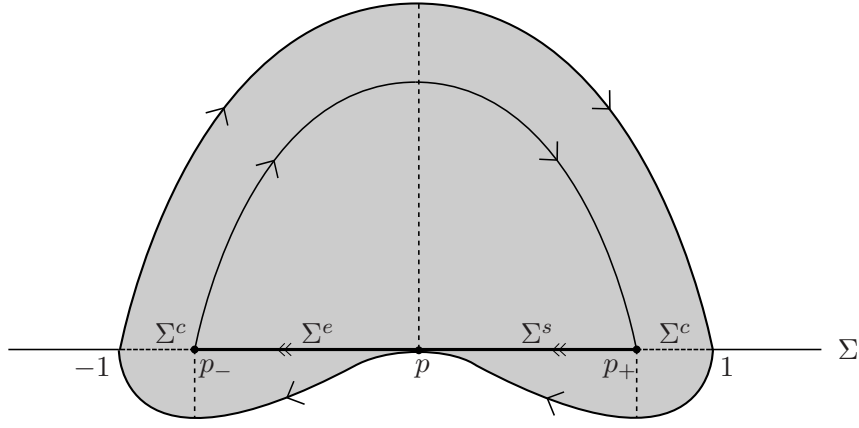}
\put(73.2,6){$p_{+}$} \put(23.5,6){$p_{-}$} \put(50,6){$p$} \put(36,9.6){$\Sigma^{e}$} \put(63,9.6){$\Sigma^{s}$} \put(18,9.6){$\Sigma^{c}$} \put(80,9.6){$\Sigma^{c}$} \put(102,7.6){$\Sigma$} \put(8,5.5){$-1$}  \put(87.5,5.5){$1$}
\end{overpic}
\caption{\small{Special integral curves and tangency points.}} \label{fig minimal}
\end{center}
\end{figure}

\end{example}

\begin{proposition}\label{minimal} Consider $Z=(X,Y) \in \Omega$, where $X(x,y)=(1,-2x)$, $Y(x,y)=(-2,4x^{3}-2x)$ and $\Sigma=f^{-1}(0)=\{(x,y)\in\mathbb{R}^{2};y=0\}$. The set
\begin{equation}\label{equacao conjunto minimal}
\Lambda=\{(x,y)\in\mathbb{R}^{2};-1\leq x\leq 1\;\; \mbox{and}\;\; 1-x^{2}\leq y\leq x^{4}/2-x^{2}/2\}.
\end{equation}
 is a minimal set for $Z$.
\end{proposition}

\begin{proof}It is easy to see that $\Lambda$ is compact and has non-empty interior. Moreover, by Definition $\ref{definicao trajetorias}$, on $\partial \Lambda \setminus\{p\}$ we have uniqueness of trajectory (here $\partial B$ means the boundary of the set $B$). Note that a global trajectory of any point in $\Lambda$ meets $p$ for some time $t^{*}$. Since $p$ is a visible tangency point for $Y$ and $p\in\overline{\partial\Sigma^{e}}\cap\overline{\partial\Sigma^{s}}$, according to the fourth bullet of Definition \ref{definicao trajetorias} any trajectory passing through $p$ remain in $\Lambda$. Consequently $\Lambda$ is $Z$-invariant. Moreover, given $p_1, p_2 \in \Lambda$ the positive global trajectory by $p_1$ reaches the sliding region between $p$ and $p_+$ and slides to $p$. The negative global trajectory by $p_2$ reaches the escaping region between $p$ and $p_-$ and slides to $p$. So, there exists a global trajectory connecting $p_1$ and $p_2$. Now, let $\Lambda^{\prime} \subset \Lambda$ be a $Z$-invariant set. Given $q_1 \in \Lambda^{\prime}$ and $q_2 \in \Lambda$ since there exists a global trajectory connecting them we conclude that $q_2 \in \Lambda^{\prime}$. Therefore, $\Lambda^{\prime} = \Lambda$ and $\Lambda$ is a minimal set.
\end{proof}

Other exotic examples can be easily obtained. In Figure \ref{fig8} we show another phase portrait of a NSVF presenting a non-trivial minimal set with non-empty interior.

\begin{figure}[h!]
\begin{center}
\epsfxsize=10cm \epsfbox{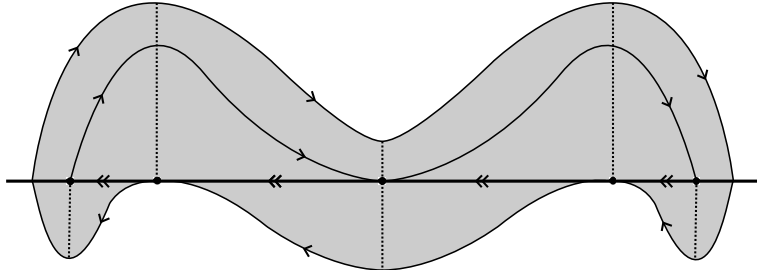} \label{fig outro minimal}
\end{center}
\caption{\small{Non-trivial minimal set presenting non-empty interior.}}\label{fig8}
\end{figure}

\section{Considerations on the Poincaré-Bendixson Theorem for NSVFs}\label{secao teo poincare-bendixson}

This section is dedicated to the Poincaré-Bendixson Theorem and present conditions in order to have a version of this important theorem of the smooth vector fields theory in the NSVFs scenario. In fact, we start considering that $\Sigma=\Sigma^{c}\cup\Sigma^{t}$, i.e., assuming that there is no sliding motion on $\Sigma$. In this case, we show that there is a similar result to the classical Poincaré-Bendixson Theorem just adding some weak hypothesis on $\Sigma$.
Additionally, we observe that some new and unpredictable phenomena in the classical theory of smooth vector fields could happen by considering that there exists sliding and escaping regions on $\Sigma$ (as we already have seen in this paper). In fact, the presence of escaping and sliding points on $\Sigma$ necessarily destroys the uniqueness of global trajectories through a point $p \in \Sigma^{s} \cup \Sigma^{e}$. This means that we can not generalize the Poincaré-Bendixson Theorem presented here without assume extra hypothesis.\\

Indeed, consider that $\Sigma=\Sigma^{c}\cup\Sigma^{t}$. In other words, $\Sigma$ has only sewing and tangential points. In this case, we can prove that an analogous to the Poincaré-Bendixson Theorem is true and we have the same cases that we had in the smooth case, adding the pseudo-cycles, the pseudo-graphs and pseudo-equilibria (see Theorem \ref{PB}, Section \ref{secao enunciados}).\\

We observe that the three first possibilities for the $\omega$-limit set of $\Gamma_{Z}(t,p)$ in Theorem \ref{PB} are related with the classical Poincaré-Bendixson Theorem. Furthermore, the other possibilities appear due to the special type of discontinuous region $\Sigma$ that we are considering (note that there are no escaping or sliding points on $\Sigma$). The proof of Theorem \ref{PB} takes into account the classical Poincaré-Bendixson Theorem and the concept of Poincaré return map for NSVFs.

\begin{proof}[Proof of Theorem \ref{PB}] Consider $p \in V$. If there exists a time $t_0>0$ such that the global trajectory $\Gamma_{Z}(t,p)$ by $p$ does not collide with $\Sigma$ for $t>t_0$ then we can apply the classical Poincaré-Bendixson Theorem in order to conclude that one of the three first cases (i), (ii) or (iii) happens. Otherwise, there exists a sequence $(t_{i})\subset\mathbb{R}$ of positive times, $t_i\rightarrow+\infty$, such that $p_i=\Gamma_{Z}(t_{i},p)\in\Sigma$.

The hypothesis that we do not have sliding motion implies $Xf(p_i)\cdot Yf(p_i)\geq0$. We observe that if $Xf(p_i)=0$ and $Yf(p_i)\neq0$ (resp., $Xf(p_i)\neq0$ and $Yf(p_i)=0$) then the trajectory of $X$ (resp., $Y$) passing through $p_i$ has an odd contact with $\Sigma$. For each $i\in\N$ we say that $p_i\in T(p)$ if one of the following cases happens: (i) $Xf(p_i)\cdot Yf(p_i) >0$, (ii) $Xf(p_i)=0$ and $Yf(p_i)\neq0$, (iii) $Xf(p_i)\neq0$ and $Yf(p_i)=0$ or (iv) $Xf(p_i)=Yf(p_i)=0$ and both have an odd contact order with $\Sigma$. If $Xf(p_i)=Yf(p_i)=0$  and the contact order of both is even then we say that $p_i\in N(p)$. Observe that, by hypothesis,   $N(p)$ is a finite set. We separate the proof in two cases: $T(p)$ is finite and $T(p)$ is not finite.

Assume that $T(p)$ is a finite set. We denote by $n_p$ and $t_p$ the number of elements of the sets $N(p)$ and $T(p)$ respectively. According to Definition \ref{definicao trajetorias}, a global trajectory of $Z$ by $p_l \in N(p)$ can follows one of two distinct paths. Let us denote by $\Gamma_m$ an arc of $\Gamma_Z(t,p)$ connecting two consecutive points $p_i$ and $p_{i+1}$, $i \in \N$. In this case there exists at most $2^{n_p} + t_p$ arcs $\Gamma_m$ of $\Gamma_{Z}(t,p)$. So, there exists a (sub)set $\Upsilon \subset \{ 1,2, \hdots, 2^{n_p} + t_p \}$ such that $\Gamma = \bigcup_{j\in \Upsilon} \Gamma_j$ is a closed orbit  intersecting $\Sigma$ (i.e., a pseudo-cycle) contained in $\Gamma_{Z}(t,p)$ and with the property that $\Gamma_{Z}(t,p)$ visit each arc $\Gamma_j$ of $\Gamma$ an infinite number of times. In what follows we prove that $\omega(\Gamma_{Z}(t,p)) = \Gamma$. In fact, as $\Gamma_{Z}(t,p)$ must visit each arc $\Gamma_j$ of $\Gamma$ an infinite number of times then $\Gamma\subset \omega(\Gamma_{Z}(t,p))$. On the other hand, if $x_0\in\omega(\Gamma_{Z}(t,p))$ then there exists a sequence $(s_k)\subset\R$, $s_k\rightarrow+\infty$, such that $\Gamma_Z(s_k,p)=x_k\rightarrow x_0$. Moreover, since $\Gamma_{Z}(t,p)$ also is composed by a finite number of arcs $\Gamma_m$, $s_k\rightarrow+\infty$ and $\Gamma_{Z}(t,p)$ has no equilibria (otherwise it does not visit $\Sigma$ infinitely many times), there exists a subsequence $(x_{k_{j}})$ of  $(x_{k})$ that visits some arcs $\Gamma_m$ infinitely many times. Since  $\Gamma$ includes all arcs $\Gamma_j$ for which the global trajectory visit $\Gamma_j$ for an infinite sequence of times,  $x_{k_{j}}\in\Gamma$ a compact set, and consequently $x_0\in\Gamma$.

Now assume that $T(p)$ in not a finite set. In this case, there exist a point $q \in \Sigma$ and a subsequence $(t_{i_j})=(s_j)$ of $(t_i)$ such that
\begin{equation}\label{fix}
\displaystyle\lim_{j\to\infty}\Gamma_{Z}(s_{j},p)=q
\end{equation}
since $\Gamma^{+}_{Z}(t,p) \subset K$, a compact set. Observe that $q\in \omega(\Gamma_{Z}(t,p))\cap\Sigma\neq\emptyset.$  As we do not have sliding motion, for each $x\in\omega(\Gamma_{Z}(t,p))\cap\Sigma$, we have only two options for it: either $x$ is a singular tangency or $x$ is a regular  point.

If there exists $x_0\in\omega(\Gamma_{Z}(t,p))\cap\Sigma$  a  singular tangency then $\omega(\Gamma_{Z}(t,p)) = \{ x_0 \}$. In fact, when both $X$ and $Y$ have an invisible tangency point at $x_0$ and there exists a sequence $(s_k)\subset\R$, $s_k\rightarrow+\infty$, such that $\Gamma_Z(s_k,p)=x_k\rightarrow x_0$ then there is a small  neighborhood $V_{x_0}$ of $x_0$ in $V$ such that all trajectory of $Z$ that starts at a point of $V_{x_0}$ converges to $x_0$. See Figure \ref{s-sing}. Therefore, $\omega(\Gamma_{Z}(t,p)) = \{ x_0 \}$ and $x_0 =q$.

\begin{figure}
\begin{center}
\begin{overpic}[height=4cm]{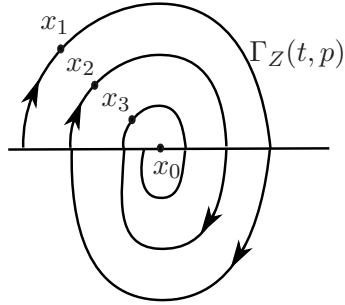}
\put(10,85){$x_1$}
\put(18,72){$x_2$}
\put(29,61){$x_3$}
\put(45,40){$x_0$}
\put(75,75){$\Gamma_{Z}(t,p)$}
\end{overpic}
\caption{\small{Case where there exists a  singular tangency in $\omega(\Gamma_{Z}(t,p))\cap\Sigma$.}} \label{s-sing}
\end{center}
\end{figure}

Suppose now that all points in $\omega(\Gamma_{Z}(t,p))\cap\Sigma$ are regular ones. Again we separate the analysis in two cases: either $\omega(\Gamma_{Z}(t,p))$ contains equilibria or contains no equilibria. Consider the case when $\omega(\Gamma_{Z}(t,p))$ contains no equilibria. Let $q$ as in Equation \eqref{fix}. If $q\in T(q)$ then it is clear that the local trajectory passing through $q$ is unique and $\Gamma_{Z}(\varepsilon,q) \in \omega(\Gamma_{Z}(t,p))$  for $\varepsilon>0$ sufficiently small. If  $q\in N(q)$ then, since $q$ can not be a singular tangency, $q$ is a visible tangency for both $X$ and $Y$. So, there are two possible choices for the positive local trajectory of $Z$ passing through $q$ and at least one of them
  is such that it is contained in  $\omega(\Gamma_{Z}(t,p))$. By continuity, the global trajectory $\Gamma(t,q)$ of $Z$ that passes through $q$, contained in $\omega(\Gamma_{Z}(t,p))$, must come back to a neighborhood $V_q$ of $q$ in $\Sigma$. The late affirmation is true, because if it does not come back then it remains in $\Sigma^+$ or in $\Sigma^-$. So, the set $\omega(\Gamma(t,q))$ is a periodic orbit of $X$ or $Y$, because there are no singular points in  $\omega(\Gamma_{Z}(t,p))$. But it is a contradiction with the fact that the orbit $\Gamma_{Z}(t,p)$ must visit any neighborhood of $q$ infinitely many times. Moreover, by the Jordan Curve Theorem, $\Gamma(t,q)\cap V_q = \{ q \}$, otherwise there exists a flow box not containing $q$ for which $\Gamma(t,q)$ and, consequently, $\Gamma(t,p)$, do not depart it. This is a contradiction with the fact that the orbit $\Gamma_{Z}(t,p)$ must visit any neighborhood of $q$ infinite many times. Therefore, $\Gamma_Z(t,q)$ is closed (i.e., is a pseudo-cycle) and $\omega(\Gamma_{Z}(t,p))=\Gamma_Z(t,q)$.

The remaining case is when $\omega(\Gamma_{Z}(t,p))$ has equilibria either of $X$ or of $Y$. In this case for each regular point $q\in\omega(\Gamma_{Z}(t,p))$ consider the local orbit $\Gamma_{Z}(t,q)$ which is contained in $\omega(\Gamma_{Z}(t,p))$. The set $\omega(\Gamma_{Z}(t,q))$ can not be a periodic orbit or a graph contained in $\Sigma^+$ or in $\Sigma^-$, because the orbit $\Gamma_{Z}(t,p)$ must visit any neighborhood of $q$ infinite many times. So, the unique option is that $\omega(\Gamma_{Z}(t,q))=\{z_i\}$ where $z_i$ is an equilibrium of $X$ or of $Y$. Similarly, the $\alpha$-limit set $\alpha(\Gamma_{Z}(t,q))=\{z_j\}$ where $z_j$ is an equilibrium of $X$ or of $Y$.  Thus, with an appropriate ordering of the equilibria $z_k$, $k=1,2\dots,m$, (which may not be distinct) and regular orbits $\Gamma_k\subset\omega(\Gamma_{Z}(t,p))$, $k=1,2\dots,m$, we have
\[ \alpha(\Gamma_k)=z_k\quad\mbox{ and }\quad \omega(\Gamma_k)=z_{k+1}\]
for $k=1,\dots,m$, where $z_{m+1}=z_1$. It follows that the global trajectory $\Gamma_Z(t,p)$ either spirals down to or out toward $\omega(\Gamma_Z(t,p))$ as $t\rightarrow+\infty$. It means that in this case  $\omega(\Gamma_Z(t,p))$ is a pseudo-graph composed by the equilibria $z_k$ and the arcs $\Gamma_k$ connecting them, $k=1,\dots,m$.

This concludes the proof of Theorem \ref{PB}.\end{proof}

Now we perform the proof of Corollary \ref{corolario teo Poinc-Bend} (see Section \ref{secao enunciados}). In Example \ref{exemplo multiplos cjtos limite} below we illustrate its consequences.\\

\begin{proof}[Proof of Corollary \ref{corolario teo Poinc-Bend}]
In fact, since by Definition \ref{limitset} the $\omega$-limit set of a point is the union of the $\omega$-limit set of all global trajectories passing through it, the conclusion is obvious.
\end{proof}

\begin{example}\label{exemplo multiplos cjtos limite}

Consider Figure \ref{fig exemplo multiplos cjtos limite}. Here we observe a NSVF without sliding motion on $\Sigma$ where the conclusions of Theorem \ref{PB} and Corollary \ref{corolario teo Poinc-Bend} can be observed. Since the uniqueness of trajectories by $p$ is not achieved (neither for positive nor for negative times) both the $\alpha$ and the $\omega$-limit sets are disconnected sets. The $\alpha$-limit set of $p$ is composed by the focus $\alpha_1$ and the s-singular tangency point $\alpha_2$. The $\omega$-limit set of $p$ is composed by the saddle $\omega_1$ and the periodic orbit $\Gamma_1$.

\begin{figure}[!h]
\begin{center}
\psfrag{A}{$p$} \psfrag{B}{$\alpha_1$} \psfrag{C}{$\omega_1$}
\psfrag{D}{$\Gamma_1$} \psfrag{E}{$\alpha_2$}
\psfrag{F}{\hspace{.7cm}$\Sigma^s$} \psfrag{G}{} \epsfxsize=9cm
\epsfbox{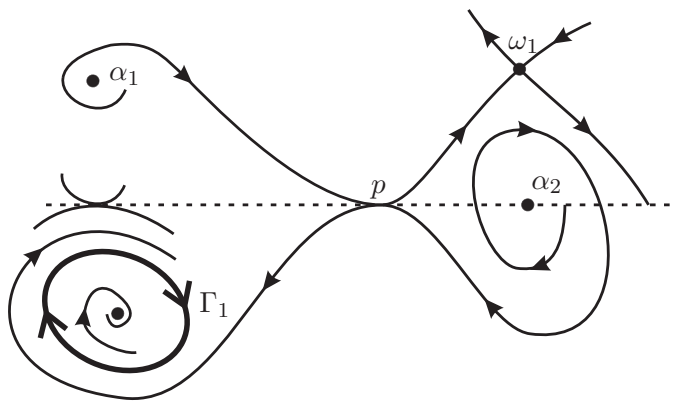} \caption{\small{Both the $\alpha$-limit set $\{ \alpha_1 , \alpha_2\}$ and the $\omega$-limit set $\{ \omega_1 , \Gamma_1 \}$ of the point $p$ are disconnected. Sliding motion on $\Sigma$ is not allowed.}} \label{fig exemplo multiplos cjtos limite}
\end{center}
\end{figure}

\end{example}

When we consider escaping and/or sliding points on $\Sigma$ (i.e., $\Sigma^s \cup \Sigma^e \neq \emptyset$), each subset $N\subset\Sigma^{e}\cup\Sigma^{s}$ is necessarily not invariant, because on this region there is no uniqueness of solution. Actually, if we take a point $q\in N$, there will exist infinitely
many solutions passing through $q$ when the time goes to future or past (see Examples \ref{exemplo varios conj limite} and \ref{exemplo multiplos cjtos limite} and Remark \ref{remark 2}).
For this reason, it is possible to occur some interesting phenomena where classical properties of both limit and minimal sets do not work. In particular we do not have a version of the Poincaré-Bendixson Theorem as stated in the classical theory for this scenario.\\

For instance, it is easy to see that the minimal set $\Lambda$ showed in Example \ref{exemplo1}
and in \eqref{equacao conjunto minimal} has interior non-empty, differently of the classical result
in smooth vector fields theory where we have the opposite situation. Nevertheless, $\Lambda$ does not look like neither a cycle nor an equilibrium,
i.e., $\Lambda$ is a non-trivial minimal set. The following proposition point out another surprising aspect of the minimal set $\Lambda$, not predicted for the classical theory.

\begin{proposition}\label{proposicao trajetoria nao densa}
Let $\Lambda$ given by \eqref{equacao conjunto minimal}. If $q \in \Lambda$ then there exists a trajectory passing trough $q$ that is not dense in $\Lambda$.
\end{proposition}\begin{proof}
Observe Figure \ref{fig minimal}. By Definition \ref{definicao trajetorias}, there exists a global trajectory $\Gamma_0$ of $Z$ with coincides to the closed curve $\partial \Lambda$, the boundary of $\Lambda$. Moreover, as shown at the proof of Proposition \ref{minimal}, given an arbitrary point $q \in \Lambda$, each global orbit passing through $q$ also reaches $p=(0,0)$ in finite time. Let $\Gamma_1$ be an arc of trajectory of $Z$ joining $q$ and $p$. So, $\Gamma = \Gamma_0 \cup \Gamma_1$ is a non-dense trajectory of $Z$ in $\Lambda$ passing through $q \in \Lambda$.
\end{proof}

\noindent {\textbf{Acknowledgments.} The first author is partially supported by a FAPESP-BRAZIL grant
2007/06896-5. The second author is partially supported by FAPESP-BRAZIL grant 2012/00481-6. The third author is supported by the FAPESP-BRAZIL grants 2010/18015-6 and 2012/05635-1}

\end{document}